\documentclass[10pt,a4paper]{amsart}
\usepackage{amsmath,amsthm,mathtools,amssymb,amsfonts}
\usepackage[utf8]{inputenc}
\usepackage[T1]{fontenc}
\usepackage{lmodern}

\usepackage{enumitem}
\setlist[enumerate,1]{label=\textup{(\arabic*)}}

\usepackage{tikz}
\usetikzlibrary{matrix,calc,arrows,decorations.pathmorphing}
\tikzset{node distance=2cm, auto}

\usepackage[lite]{amsrefs}

\renewcommand*{\PrintDOI}[1]{\href{http://dx.doi.org/\detokenize{#1}}{doi: \detokenize{#1}}}

\BibSpec{thesis}{%
    +{}  {\PrintAuthors}                {author}
    +{,} { \textit}                     {title}
    +{:} { \textit}                     {subtitle}
    +{,} { \PrintThesisType}            {type}
    +{,} { }                            {organization}
    +{,} { }                            {address}
    +{,} { \PrintDateB}                 {date}
    +{,} { \eprint}                     {eprint}
    +{,} { }                            {status}
    +{}  { \parenthesize}               {language}
    +{}  { \PrintTranslation}           {translation}
    +{;} { \PrintReprint}               {reprint}
    +{.} { }                            {note}
    +{.} {}                             {transition}
    +{} { \PrintDOI}                   {doi}
    +{} { available at \url}            {eprint}
    +{}  {\SentenceSpace \PrintReviews} {review}
}

\BibSpec{book}{%
  +{}  {\PrintPrimary}                {transition}
  +{,} { \textit}                     {title}
  +{.} { }                            {part}
  +{:} { \textit}                     {subtitle}
  +{,} { \PrintEdition}               {edition}
  +{}  { \PrintEditorsB}              {editor}
  +{,} { \PrintTranslatorsC}          {translator}
  +{,} { \PrintContributions}         {contribution}
  +{,} { }                            {series}
  +{,} { \voltext}                    {volume}
  +{,} { }                            {publisher}
  +{,} { }                            {organization}
  +{,} { }                            {address}
  +{,} { \PrintDateB}                 {date}
  +{,} { }                            {status}
  +{}  { \parenthesize}               {language}
  +{}  { \PrintTranslation}           {translation}
  +{;} { \PrintReprint}               {reprint}
  +{.} { }                            {note}
  +{.} {}                             {transition}
  +{} { \PrintDOI}                   {doi}
  +{} { available at \url}            {eprint}
  +{}  {\SentenceSpace \PrintReviews} {review}
}

\usepackage[pdftitle={Bicategories in noncommutative geometry},
pdfauthor={Ralf Meyer},
pdfsubject={Mathematics}
]{hyperref}
\usepackage{microtype}

\newtheorem{lemma}{Lemma}[section]
\newtheorem{proposition}[lemma]{Proposition}
\newtheorem{theorem}[lemma]{Theorem}
\theoremstyle{definition}
\newtheorem{definition}[lemma]{Definition}
\theoremstyle{remark}
\newtheorem{example}[lemma]{Example}

\newcommand{\longref}[2]{\hyperref[#2]{#1~\textup{\ref*{#2}}}}

\newcommand*{\nb}{\nobreakdash}% for non-breakable hyphens
\newcommand*{\alb}{\hspace{0pt}}% allow break in word after hyphen

\newcommand*{\Star}{$^*$\nb-\alb{}} % used for *-homomorphisms
\newcommand*{\Cst}{\mathrm C^*}% C*-algebra

\newcommand*{\defeq}{\mathrel{\vcentcolon=}}% used for definitions
\newcommand{\congto}{\xrightarrow{\simeq}}% isomorphism as map

\newcommand{\full}{\mathrm{full}}% subscript for full correspondences
\newcommand{\prop}{\mathrm{prop}}% subscript for proper correspondences and morphisms
\newcommand{\strict}{\mathrm{str}}% subscript for strictly continuous maps
\newcommand{\assoc}{\mathrm{ass}} % associator
\newcommand{\id}{\mathrm{id}}
\newcommand{\injto}{\hookrightarrow}% injective map

\DeclarePairedDelimiter{\abs}{\lvert}{\rvert}% absolute value
\DeclarePairedDelimiter{\ket}{\lvert}{\rangle}% ket-bra notation
\DeclarePairedDelimiterX{\braket}[2]{\langle}{\rangle}{#1\,\delimsize\vert\,\mathopen{}#2}% inner product

\DeclareMathOperator{\Map}{Map} % space of continuous maps

\newcommand{\N}{\mathbb N} % natural numbers (including 0)
\newcommand{\Comp}{\mathbb K} % compact operators

\newcommand*{\Hilm}[1][E]{\mathcal #1}% Hilbert module
\newcommand{\Mult}{\mathcal M}% multiplier algebra
\newcommand{\Cat}[1][C]{\mathcal #1} % category C
\newcommand*{\Cstcat}{\mathcal C^*} % C*-algebra category
\newcommand{\Corr}{\mathfrak{Corr}} % correspondence quasi-category

\newcommand*{\Simp}[1]{{\vartriangle^{\!#1}}} % #1-simplex
\newcommand*{\Subs}[1]{\mathrm{Sub}{[#1]}}% subsets of [#1]
\newcommand*{\Subdiv}[1]{\operatorname{Bar}{\vartriangle^{\!#1}}}% barycentric subdivision of #1-simplex
\newcommand*{\CSubdiv}[1]{\operatorname{\overline{Bar}}{\vartriangle^{\!#1}}}% variant of Subdiv used in proof of universal property of correspondence categories

\tikzset{cdar/.style=->,auto}
\tikzset{dar/.style={double,double equal sign distance,-implies}}%style for 2-arrows in triangles
\tikzset{mid/.style={anchor=mid}} % put labels on the arrow
\tikzset{deq/.style={double,double equal sign distance,-}}%style for 2-arrows in triangles

\tikzset{triar/.style={anchor=mid,->}}%style for 1-arrows in triangles
\tikzset{tridar/.style={anchor=mid,double,double equal sign
distance,-implies}}%style for 2-arrows in triangles
\tikzset{narrowfill/.style={inner sep=0pt, fill=white}}% style for nodes with filled background

\newcommand*{\twotriangle}[8][1]{\begin{tikzpicture}[scale=#1,baseline=(current bounding box.west)]
    \node (0) at (210:1) {$#2$};
    \node (1) at (90:1) {$#3$};
    \node (2) at (330:1) {$#4$};
    \draw[triar] (0) -- node[fill=white] (01) {$#5$} (1);
    \draw[triar] (1) -- node[fill=white] (12) {$#6$} (2);
    \draw[triar] (0) -- node[fill=white] (02) {$#7$} (2);
    \draw[tridar](1) -- node[near end,narrowfill] (012) {$#8$} (02);
  \end{tikzpicture}}

\begin{document}
\title{Bicategories of C*-correspondences as Dwyer--Kan localisations}
\author{Ralf Meyer}
\email{rmeyer2@uni-goettingen.de}
\address{Mathematisches Institut\\
  Universit\"at G\"ottingen\\
  Bunsenstra\ss e 3--5\\
  37073 G\"ottingen\\
  Germany}



\begin{abstract}
  We show that the bicategory of proper correspondences is the
  Dwyer--Kan localisation of the category of C*-algebras at a certain
  class of \Star{}homomorphisms.
\end{abstract}
\maketitle

\section{Introduction}
\label{sec:intro}

The correspondence bicategory is a bicategory that has
\(\Cst\)\nb-algebras as objects, correspondences as arrows, and
isomorphisms of correspondences as \(2\)\nb-arrows.  In this article,
we use the conventions in~\cite{Buss-Meyer-Zhu:Higher_twisted}, that
is, an \(A,B\)-correspondence is an arrow from~\(A\) to~\(B\) and the
composition is the balanced tensor product of
\(\Cst\)\nb-correspondences in reverse order.  The correspondence
bicategory was first introduced to study Morita--Rieffel equivalences
of \(\Cst\)\nb-algebras because the latter are just the equivalences
in this bicategory (see \cites{Brouwer:Bicategorical,
  Landsman:Bicategories}).  It allows to formulate what it means for a
group to act on a \(\Cst\)\nb-algebra by Morita--Rieffel equivalences,
and such an action turns out to be the same as a saturated Fell bundle
(see~\cite{Buss-Meyer-Zhu:Higher_twisted}).  There is a bicategorical
variant of the usual concept of a (co)limit in a category.  The
universal property that defines a bicategorical colimit in the
correspondence bicategory is identical to the universal property of
the crossed product of a group action and equivalent to the universal
property of the Cuntz--Pimsner algebra of a proper
\(\Cst\)\nb-correspondence or a proper product system
(see~\cite{Albandik-Meyer:Colimits}).  In this way, many important
\(\Cst\)\nb-algebraic constructions are unified by viewing them as
colimits of different diagrams in the bicategory of proper
correspondences.  This
point of view allows to clarify the functoriality properties of
relative Cuntz--Pimsner algebras
(see~\cite{Meyer-Sehnem:Bicategorical_Pimsner}).  The correspondence
bicategory has a groupoid analogue, and this allows to study groupoid
models for important examples of Cuntz--Pimsner algebras such as the
\(\Cst\)\nb-algebras of (self-similar, topological, higher-rank)
graphs (see \cites{Albandik:Thesis, Albandik-Meyer:Product,
  Antunes-Ko-Meyer:Groupoid_correspondences,
  Meyer:Groupoid_models_relative}).

In higher category theory, a bicategory is only the first step from a
category to an \(\infty\)\nb-category.  Many constructions in homotopy
theory and homological algebra lead to \(\infty\)\nb-categories right
away.  Thus one may wonder whether the correspondence bicategory may
be enriched further by some even higher arrows, giving a true
\(\infty\)\nb-category.  There is, however, no obvious way to do this.
One aim of this article is to explain why this is so.  Namely, we are
going to show that the quasi-category defined by the proper
correspondences is the Dwyer--Kan localisation of the
category of \(\Cst\)\nb-algebras and \Star{}homomorphisms at the class
of corner embeddings.  This is remarkable because Dwyer--Kan
localisation is a construction whose output is usually an
\(\infty\)\nb-category.  In this case, however, this
\(\infty\)\nb-category is just a \(2\)\nb-category, that is, all
higher arrows are trivial.  Our result gives a universal property of
the proper correspondence bicategory: a functor from the category of
\(\Cst\)\nb-algebras and \Star{}homomorphisms to a quasi-category or
even to an \(\infty\)\nb-category factors through the proper correspondence
bicategory if and only if it is \(\Cst\)\nb-stable, and this
factorisation is unique up to equivalence if it exists.  This
universal property improves upon a known universal property of the
truncation of the proper correspondence bicategory to a category (see
\cite{Meyer:KK-survey}*{Proposition~39}).

The description of the proper correspondence bicategory as a
Dwyer--Kan
localisation also clarifies its relationship to recent constructions
of \(\infty\)\nb-category enrichments of Kasparov's bivariant
KK-theory, which use Dwyer--Kan localisation as well.  Recall that
Kasparov's bivariant KK-theory, viewed as a category, is characterised
uniquely by a universal property.  Namely, the canonical functor from
the category of separable \(\Cst\)\nb-algebras and
\Star{}homomorphisms to the KK-category is the universal
\(\Cst\)\nb-stable, split exact, homotopy invariant functor.  A
functor is \(\Cst\)\nb-stable, split exact, or homotopy invariant if
and only if it maps certain \Star{}homomorphisms to isomorphisms.
Thus KK-theory is a localisation of the category of
\(\Cst\)\nb-algebras.  Recently, \(\infty\)\nb-categorical enrichments
of KK-theory have been constructed by Bunke, Engel and
Land~\cite{Bunke-Engel-Land:Stable_infinity-category_KK} using
Dwyer--Kan localisation.  This leads to a stable
\(\infty\)\nb-category whose homotopy category is equivalent to
KK-theory as a triangulated category.

In this article, we describe \(\infty\)\nb-categories as
quasi-categories, that is, simplicial sets in which all inner horns
may be filled (see \cite{Joyal:Quasi-categories}*{Definition~1.1}).
The nerve of a category or of a bicategory with invertible
\(2\)\nb-arrows is a quasi-category in this sense.  Thus we view the
category of \(\Cst\)\nb-algebras and the correspondence bicategory as
quasi-categories.  A functor of categories or a homomorphism of
bicategories induces a simplicial map between their nerves, and an
equivalence of functors or bicategory homomorphisms induces a homotopy
between these maps.  Our main theorem describes homotopy classes of
simplicial maps from the nerve of the proper correspondence bicategory
to any other quasi-category~\(\Cat[D]\) as homotopy classes of
simplicial maps from the category of \(\Cst\)\nb-algebras
to~\(\Cat[D]\)
that map all corner embeddings to equivalences in~\(\Cat[D]\).  This
universal property is also used to define the Dwyer--Kan localisation
of a quasi-category at a family of morphisms.

Our main theorem treats only proper correspondences because it starts
with \Star{}homomorphisms.  In the final section, we briefly discuss
variants of it that lead to the bicategories
of all correspondences or all full correspondences.  These start with
(full) strictly continuous \Star{}homomorphisms on multiplier
algebras.  In the following, we state our results in the setting of
\(\Cst\)\nb-algebras without extra structure for simplicity.  We may,
however, generalise all our arguments to \(\Cst\)\nb-algebras with
extra structure such as an action of a locally compact groupoid, a
quantum group, or various other kinds of extra structure.  For these
variants, we merely have to ask all \Star{}homomorphisms and
\(\Cst\)\nb-correspondences to be equivariant with respect to the
chosen extra structure.

This article is based on joint work with Rohan Lean a couple of years
ago, which was intended for a book project which is taking an awfully
long time to complete.  I decided that the results in this article
deserve to be
published now because of recent work on Dwyer--Kan localisations of
the category of \(\Cst\)\nb-algebras.

\section{The nerve of the proper correspondence bicategory}
\label{sec:nerve_Corr}

Let~\(\Cstcat_+\) be the category of \(\Cst\)\nb-algebras and
\Star{}homomorphisms.  Let \(N\Cstcat_+\) be its nerve.  This
is~\(\Cstcat_+\) viewed as a quasi-category.  Let \(\Corr_\prop\)
denote the bicategory of proper correspondences
(see~\cite{Buss-Meyer-Zhu:Higher_twisted}, where bicategories are
called weak \(2\)\nb-categories).  It has \(\Cst\)\nb-algebras as
objects, proper \(A,B\)-correspondences as arrows \(A\to B\), and
isomorphisms of \(A,B\)-correspondences as \(2\)\nb-arrows (called
bigons in~\cite{Buss-Meyer-Zhu:Higher_twisted}).  The composition of
arrows is the balanced tensor product of \(\Cst\)\nb-correspondences
in reverse order.  Beware that I am using the conventions
of~\cite{Buss-Meyer-Zhu:Higher_twisted} here, although in more recent
articles, I have advocated to view an \(A,B\)-correspondence as an
arrow from~\(B\) to~\(A\).  It is important to allow only isomorphisms
of \(\Cst\)\nb-correspondences as \(2\)\nb-arrows as
in~\cite{Buss-Meyer-Zhu:Higher_twisted} in order to turn this
bicategory into a quasi-category.

The nerve~\(N\Cat\) of a bicategory~\(\Cat\) is described by
Duskin~\cite{Duskin:Simplicial}, generalising the nerve of a category.
Duskin also proves that it is a quasi-category
(see~\cite{Duskin:Simplicial}).  Duskin's definition leads to the
following description of the nerve:

\begin{definition}
  \label{def:NCorr}
  Let \(N\Corr_\prop\) be the nerve of the bicategory of proper
  correspondences.  An \(n\)\nb-simplex in~\(N\Corr_\prop\) is given
  by \(\Cst\)\nb-algebras \(A_i\) for \(0\le i \le n\),
  \(\Cst\)\nb-correspondences \(\Hilm_{i j}\colon A_i\to A_j\) for
  \(0\le i \le j \le n\), and isomorphisms of
  \(\Cst\)\nb-correspondences
  \(u_{i j k}\colon \Hilm_{i j} \otimes_{A_j} \Hilm_{j k}\to \Hilm_{i
    k}\) for \(0 \le i \le j \le k \le n\), such that
  \begin{enumerate}
  \item \(\Hilm_{ii}\) is the identity correspondence~\(A_i\) and
    \(u_{i i k}\) and \(u_{i k k}\) are the canonical isomorphisms
    \(A_i \otimes_{A_i} \Hilm_{i k} \congto \Hilm_{i k}\) and
    \(\Hilm_{i k} \otimes_{A_k} A_k \congto \Hilm_{i k}\) for
    \(0 \le i \le k \le n\);
  \item if \(0 \le i \le j \le k \le l \le n\), then the following
    diagram of isomorphisms of \(\Cst\)\nb-correspondences commutes:
    \begin{equation}
      \label{eq:n-simplex_correspondence}
      \begin{tikzpicture}[yscale=1.2,xscale=3,baseline=(current
        bounding box.west)]
        \node (12) at (144:1)
        {\(\Hilm_{ij}\otimes_{A_j}(\Hilm_{jk}\otimes_{A_k}\Hilm_{kl})\)};
        \node (12a) at (216:1)
        {\((\Hilm_{ij}\otimes_{A_j}\Hilm_{jk})\otimes_{A_k}\Hilm_{kl}\)};
        \node (2) at (72:1) {\(\Hilm_{ij}\otimes_{A_j}\Hilm_{jl}\)};
        \node (e) at (0:.8) {\(\Hilm_{il}\)}; \node (1) at (288:1)
        {\(\Hilm_{ik}\otimes_{A_k}\Hilm_{kl}\)};
        \draw[deq] (12) -- node[swap]
        {\(\assoc\)} (12a); \draw[dar] (12.north) -- node
        {\(\id_{\Hilm_{ij}}\otimes_{A_j}u_{jkl}\)} (2);
        \draw[dar] (12a.south) -- node[swap]
        {\(u_{ijk}\otimes_{A_k}\id_{\Hilm_{kl}}\)} (1);
        \draw[dar] (2.south east) -- node {\(u_{ijl}\)} (e);
        \draw[dar] (1.north east) -- node[swap] {\(u_{ikl}\)} (e);
      \end{tikzpicture}
    \end{equation}
  \end{enumerate}
  An increasing map \(\varphi \colon [n] \to [m]\) induces a map from
  the set of \(m\)\nb-simplices to the set of \(n\)\nb-simplices by
  mapping \((A_i,\Hilm_{i j},u_{i j k})\) to
  \((A_{\varphi(i)}, \Hilm_{\varphi(i) \varphi(j)}, u_{\varphi(i)
    \varphi(j) \varphi(k)})\).
\end{definition}

A \Star{}homomorphism \(\varphi\colon A\to B\) induces a proper
\(A,B\)\nb-correspondence~\(\Gamma_\varphi\) as follows: its
underlying Hilbert \(B\)\nb-module is the right ideal
\(\Gamma_\varphi\defeq \varphi(A)\cdot B\subseteq B\), with~\(A\)
acting nondegenerately on the left by \(a\cdot b \defeq \varphi(a) b\)
for \(a\in A\), \(b\in \Gamma_\varphi\).  For two composable
\Star{}homomorphisms \(\varphi\colon A\to B\) and \(\psi\colon B\to
C\), there is a canonical isomorphism
\begin{equation}
  \label{eq:Gamma_multiplicative}
  \Gamma_\varphi \otimes_B \Gamma_\psi
  \cong \Gamma_{\psi\circ \varphi},\qquad
  b\otimes c \mapsto \psi(b)c;
\end{equation}
this map is isometric, and
\(\psi\circ \varphi(A)\psi(B)C =\psi\circ \varphi(A)C\) implies that
it is surjective.  The isomorphisms in~\eqref{eq:Gamma_multiplicative}
make
\(\varphi\mapsto \Gamma_\varphi\) into a homomorphism of bicategories
\(\Cstcat_+ \to \Corr_\prop\).  So it induces a simplicial map
\(\Gamma\colon N\Cstcat_+ \to N\Corr_\prop\).

\section{\texorpdfstring{$\Cst$}{C*}-stable functors}
\label{sec:stable_fun}

Let~\(\Cat[D]\) be a quasi-category.  Let \(F\colon N\Cstcat_+\to
\Cat[D]\) be a simplicial map, which we consider as a functor
\(\Cstcat_+\to \Cat[D]\).

\begin{definition}
  \label{def:Cstar-stable_functor}
  For a \(\Cst\)\nb-algebra~\(B\) and a Hilbert
  \(B\)\nb-module~\(\Hilm\), let \(\Hilm\oplus B\) be the direct sum
  of Hilbert modules.  We define the \emph{corner embedding}
  for~\(\Hilm\) to be the canonical inclusion map
  \[
    i_{\Hilm}\colon B\cong\Comp(B)\to \Comp(\Hilm\oplus B).
  \]
  We call a functor \(F\colon N\Cstcat_+\to \Cat[D]\)
  \emph{\(\Cst\)\nb-stable} if \(F(i_{\Hilm})\) is an equivalence for
  each Hilbert module~\(\Hilm\).
\end{definition}

\begin{example}
  The canonical functor \(\Gamma\colon N\Cstcat_+\to N\Corr_\prop\) is
  \(\Cst\)\nb-stable.  To see this, we first identify
  \(\Gamma(i_{\Hilm})\) with the proper correspondence
  \(\Comp(\Hilm\oplus B,B)\).  This is an equivalence, its
  quasi-inverse is \(\Comp(B,\Hilm\oplus B)\).  That is,
  \[
    \Comp(\Hilm\oplus B,B) \otimes_B \Comp(B,\Hilm\oplus B)
    \cong \Comp(\Hilm\oplus B),\qquad
    \Comp(B,\Hilm\oplus B) \otimes_{\Comp(\Hilm\oplus B)}
    \Comp(\Hilm\oplus B,B) \cong B.
  \]
\end{example}

We may also speak of \(\Cst\)\nb-stable functors on subcategories
of~$\Cstcat_+$ such as the subcategory of \(\sigma\)\nb-unital or of
separable
\(\Cst\)\nb-algebras.  In these cases, we may simplify the definition,
replacing the corner embeddings for general Hilbert modules by a
smaller class of ``special'' corner embeddings:

\begin{lemma}
  A functor on the category of \(\sigma\)\nb-unital
  \(\Cst\)\nb-algebras is \(\Cst\)\nb-stable if and only if it maps
  the special corner embeddings \(B\to \Comp(\ell^2\N)\otimes B\),
  \(b\mapsto E_{00} \otimes b\), to equivalences.
\end{lemma}

\begin{proof}
  The special corner embeddings are the corner embeddings for the
  Hilbert modules \(\Comp(\ell^2(\N_{\ge1})\otimes B)\), and so a
  \(\Cst\)\nb-stable functor maps them to equivalences.  Conversely,
  let us assume that the functor~\(F\) maps the special corner
  embeddings to equivalences.  Let~\(B\) be a \(\sigma\)\nb-unital
  \(\Cst\)\nb-algebra and let~\(\Hilm\) be a Hilbert module such that
  \(\Comp(\Hilm\oplus B)\) is also \(\sigma\)\nb-unital.  This is
  equivalent to \(\Hilm\oplus B\) being countably generated as a
  Hilbert module over~\(B\) (see
  \cite{Lance:Hilbert_modules}*{Proposition~6.7}).  Then Kasparov's
  Stabilisation Theorem applies and shows that
  \(\Hilm \oplus \ell^2(\N_{\ge1},B) \cong \ell^2(\N_{\ge1},B)\) (see
  \cite{Lance:Hilbert_modules}*{Theorem~6.2}).  We consider the
  extension by zero maps
  \[
    B \xrightarrow{i_1} \Comp(\Hilm\oplus B)
    \xrightarrow{i_2} \Comp(\Hilm\oplus B \oplus \ell^2(\N_{\ge1},B))
    \cong \Comp(\ell^2(\N,B)).
  \]
  The composite \(i_3 \defeq i_2\circ i_1\) is a special corner
  embedding, so that \(F(i_3)\) is an equivalence.  Then
  \(F(i_2) \circ F(i_1)\) is an equivalence as well.  We may turn
  \[
    \Comp(\Hilm\oplus B, \Hilm\oplus B \oplus \ell^2(\N_{\ge1},B))
  \]
  into a Hilbert module over \(\Comp(\Hilm\oplus B)\) that contains
  \(\Comp(\Hilm\oplus B)\) as a direct summand, and the compact
  operators on this Hilbert module are \(\Comp(\Hilm\oplus B \oplus
  \ell^2(\N_{\ge1},B))\).
  Thus~\(i_2\) is also a corner embedding.
  The argument above provides another extension-by-zero map
  \(i_4\colon \Comp(\Hilm\oplus B \oplus \ell^2(\N_{\ge1},B)) \to
  \Comp((\Hilm\oplus B) \otimes \ell^2(\N))\) so that \(i_5 \defeq
  i_4\circ i_2\) is a special corner embedding, and then \(F(i_4)
  \circ F(i_2)\) is an equivalence.
  Now the Two-Out-Of-Six Lemma for the composition of \(F(i_4)\), \(F(i_2)\) and \(F(i_1)\) implies that~\(F(i_1)\) is an equivalence.
\end{proof}

For two simplicial sets \(X\) and~\(Y\), let \(\Map(X,Y)\) denote the
simplicial set of simplicial maps \(X\times\Simp{n} \to Y\).  This is
a quasi-category if~\(Y\) is one (see
\cite{Lurie:Higher_Topos}*{Proposition~1.2.7.3.(1)}, proven in
\cite{Lurie:Higher_Topos}*{Section~2.2.5}).  In particular,
\(\Map(N\Cstcat_+,\Cat[D])\) and \(\Map(N\Corr_\prop,\Cat[D])\) are
quasi-categories for any quasi-category~\(\Cat[D]\); their
\(n\)\nb-simplices are simplicial maps
\(N\Cstcat_+\times \Simp{n}\to \Cat[D]\) and
\(N\Corr_\prop\times \Simp{n}\to \Cat[D]\), respectively.  We are
interested in the full subcategory \(\Map_\Comp(N\Cstcat_+,\Cat[D])\)
of \(\Cst\)\nb-stable simplicial maps \(N\Cstcat_+\to \Cat[D]\); its
\(n\)\nb-simplices are those simplicial maps
\(N\Cstcat_+\times \Simp{n}\to \Cat[D]\) for which the adjunct
simplicial map \(N\Cstcat_+\to \Map(\Simp{n}, \Cat[D])\) is
\(\Cst\)\nb-stable.

\begin{theorem}
  \label{the:Corr_prop_universal}
  Let~\(\Cat[D]\) be a quasi-category.  The following simplicial map
  is a homotopy equivalence:
  \[
    \Gamma^*\colon \Map(N\Corr_\prop, \Cat[D]) \to
    \Map_\Comp(N\Cstcat_+,\Cat[D]),\ F\mapsto F\circ \Gamma.
  \]
\end{theorem}

This theorem is our main theorem.  It says in elementary language that
the proper correspondence quasi-category \(N\Corr_\prop\) is the
Dwyer--Kan localisation of the category of
\(\Cst\)\nb-algebras~\(\Cstcat_+\) at the class of corner embeddings.

Since~\(\Gamma\) is \(\Cst\)\nb-stable, \(F\circ \Gamma\) is
\(\Cst\)\nb-stable for any simplicial map \(F\colon N\Corr_\prop\to
\Cat[D]\).  Hence~\(\Gamma^*\) is a simplicial map to
\(\Map_\Comp(N\Cstcat_+,\Cat[D])\) as asserted.

\section{Proof of the main theorem}
\label{sec:proof_main}

The following proposition is a key first step in the proof:

\begin{proposition}
  \label{pro:extend_to_Corrprop}
  For any \(F\in \Map_\Comp(N\Cstcat_+,\Cat[D])\) there is
  \(\bar{F}\in \Map(N\Corr_\prop, \Cat[D])\) with
  \(\bar{F}\circ \Gamma=F\).
\end{proposition}

To prove this, we need the \emph{barycentric
  subdivision}~\(\Subdiv{n}\) of~\(\Simp{n}\).  This simplicial set
has nonempty subsets \(S\subseteq [n]\) as vertices and chains of
inclusions
\(S_0\subseteq S_1\subseteq S_2\subseteq S_3\subseteq \cdots\subseteq
S_n\) as \(n\)\nb-simplices.  Thus~\(\Subdiv{n}\) is the nerve of the
partially ordered set~\(\Subs{n}\) of nonempty subsets of~\([n]\),
viewed as a category by putting a unique arrow \(S\to T\) for
\(S\subseteq T\) and no arrow otherwise.

Let \((A_i,\Hilm_{i j}, u_{i j k})\) be an \(n\)\nb-simplex in the
nerve \(N\Corr_\prop\) as in Definition~\ref{def:NCorr}.  We are going
to construct a functor \(\Subdiv{n}\to N\Cstcat_+\) from it.  Let
\(S=\{i_0,\dotsc,i_k\}\subseteq [n]\) with \(i_0< \dotsb< i_k\).
Then we define a Hilbert \(A_{i_k}\)\nb-module by
\begin{equation}
  \label{eq:E_S_decomposition}
  \Hilm_S \defeq \Hilm_{i_0,i_k} \oplus
  \Hilm_{i_1,i_k} \oplus \dotsb \oplus \Hilm_{i_k,i_k},
\end{equation}
and we let \(A_S \defeq \Comp(\Hilm_S)\).

First let \(S\subseteq T\) satisfy \(\max S=\max T\).  Then both
\(\Hilm_S\) and~\(\Hilm_T\) are Hilbert modules over the same
\(\Cst\)\nb-algebra.  There is an obvious adjointable inclusion map
\(\Hilm_S\injto \Hilm_T\), which induces a \Star{}homomorphism
\(f_{ST}\colon A_S\to A_T\) by letting an operator on~\(\Hilm_S\) act
as zero on the orthogonal complement of~\(\Hilm_S\) in~\(\Hilm_T\).
By the definition of a simplex in \(N\Corr_\prop\),
\(\Hilm_{i_k,i_k}\) is the identity correspondence on~\(A_{i_k}\); thus
the embedding \(A_{i_k}\to A_S\) is a corner embedding as in
\longref{Definition}{def:Cstar-stable_functor}.

Now consider \(S\subseteq T\) with \(\max S\neq \max T\).  Set
\(j\defeq \max T\).  The unitaries \(u_{i_\ell,i_k,j}\colon
\Hilm_{i_\ell,i_k}\otimes_{A_{i_k}} \Hilm_{i_k,j}\to
\Hilm_{i_\ell,j}\) for \(\ell=0,\dotsc, k\) provide an adjointable
embedding
\[
  \Hilm_S\otimes_{A_{i_k}} \Hilm_{i_k,j} \congto \bigoplus_{\ell=0}^k
  \Hilm_{i_\ell,j} \subset \bigoplus_{i\in T} \Hilm_{i,j} = \Hilm_T.
\]
We define \(f_{ST}\colon A_S\to A_T\) by first mapping
\(x\mapsto x\otimes \id_{\Hilm_{i_k,j}}\) and then extending by zero
on the orthogonal complement of the direct summand
\(\Hilm_S\otimes_{A_{i_k}} \Hilm_{i_k,j}\)
in~\(\Hilm_T\).

\begin{lemma}
  \label{lem:fST_functor_Subdiv}
  The maps \(f_{ST}\colon A_S\to A_T\) for \(S\subseteq T\) form a
  functor \(\Subdiv{n}\to N\Cstcat_+\).
\end{lemma}

\begin{proof}
  We must show that \(f_{TU}\circ f_{ST} = f_{SU}\) for
  \(S\subseteq T\subseteq U\).  Let \(i\), \(j\) and~\(k\) be the
  maxima of \(S\), \(T\) and~\(U\), respectively.  The two maps
  \(f_{TU}\circ f_{ST}\) and~\(f_{SU}\) map the \(\Cst\)\nb-algebra of
  compact operators on the direct sum of~\(\Hilm_{l,i}\) for
  \(l\in S\) to the \(\Cst\)\nb-algebra of compact operators on the
  direct sum of~\(\Hilm_{l,k}\) for \(l\in U\).  It suffices to check
  \(f_{TU}\circ f_{ST}(x) = f_{SU}(x)\) when~\(x\) is an operator in
  \(\Comp(\Hilm_{l,i},\Hilm_{m,i})\) extended by zero on the other
  summands.  By definition, \(f_{SU}\) maps this to the operator
  \(u_{mik}(x\otimes 1_{\Hilm_{ik}})u_{lik}^*\) in
  \(\Comp(\Hilm_{l,k},\Hilm_{m,k})\), again extended by zero on the
  other summands, whereas \(f_{TU}\circ f_{ST}\) maps it to
  \[
    u_{mjk}((u_{mij}(x\otimes 1_{\Hilm_{ij}})u_{lij}^*)
    \otimes 1_{\Hilm_{jk}})u_{ljk}^* \in
    \Comp(\Hilm_{l,k},\Hilm_{m,k}),
  \]
  extended by zero.  These operators are equal
  by~\eqref{eq:n-simplex_correspondence}.
\end{proof}

\begin{lemma}
  \label{lem:subdivision_back_to_corr}
  Let~\(\Hilm\) be a proper \(A,B\)-correspondence.  Define a
  \Star{}homomorphism
  \(f_{\Hilm}\colon A\to \Comp(\Hilm)\subset \Comp(\Hilm\oplus B)\)
  from the left \(A\)\nb-action on~\(\Hilm\) and let
  \(i_{\Hilm}\colon B\to \Comp(\Hilm\oplus B)\) be the corner
  embedding.  Let \(\Gamma(f_{\Hilm})\) and \(\Gamma(i_{\Hilm})\)
  denote the associated proper correspondences \(A\to \Comp(\Hilm\oplus B)\)
  and \(B\to \Comp(\Hilm\oplus B)\).  There is a canonical isomorphism
  of correspondences
  \(u_{\Hilm}\colon \Hilm\otimes_B\Gamma(i_{\Hilm}) \cong
  \Gamma(f_{\Hilm})\).
\end{lemma}

\begin{proof}
  We may identify the underlying Hilbert
  \(\Comp(\Hilm\oplus B)\)-module of \(\Gamma(i_{\Hilm})\) with
  \[
    B\cdot\Comp(\Hilm\oplus B) \cong \Comp(\Hilm\oplus B,B).
  \]
  This becomes a correspondence isomorphism when \(B\) and
  \(\Comp(\Hilm\oplus B)\) act on \(\Comp(\Hilm\oplus B,B)\) by left
  and right multiplication, respectively.  Similarly,
  \(\Gamma(f_{\Hilm}) \cong \Comp(\Hilm\oplus B,\Hilm)\) with the
  obvious Hilbert module structure over \(\Comp(\Hilm\oplus B)\)
  and~\(A\) acting on the left by multiplication with the operator in
  \(\Comp(\Hilm)\) from the left multiplication in the proper
  \(A,B\)-correspondence~\(\Hilm\).
  
  There is an isometry
  \(\Hilm\otimes_B\Comp(\Hilm\oplus B,B) \cong \Comp(\Hilm\oplus
  B,\Hilm)\), mapping \(\xi\otimes \eta\) to the rank-one operator
  \(\ket{\xi}\circ \eta\colon \Hilm\oplus B \to B \to \Hilm\); here
  \(\ket{\xi}(b) \defeq \xi\cdot b\).  A quick calculation shows that
  this is a correspondence isomorphism
  \(\Hilm\otimes_B\Gamma(i_{\Hilm}) \cong \Gamma(f_{\Hilm})\).
\end{proof}

The diagram
\[
  A\xrightarrow{f_{\Hilm}} \Comp(\Hilm\oplus B)
  \xleftarrow{i_{\Hilm}} B
\]
in \longref{Lemma}{lem:subdivision_back_to_corr} is exactly the
functor \(\Subdiv{1}\to N\Cstcat_+\) constructed from the
arrow~\(\Hilm\) in~\(\Corr_\prop\).

Let \(\bar{F}\colon N\Corr_\prop\to \Cat[D]\) be a simplicial map to a
quasi-category that extends a \(\Cst\)\nb-stable simplicial map
\(F\colon N\Cstcat_+\to \Cat[D]\).  Then it maps the triangle for the
isomorphism~\(u_{\Hilm}\) in
\longref{Lemma}{lem:subdivision_back_to_corr} to a triangle
\begin{equation}
  \label{eq:triangle_to_fill}
  \twotriangle[2.5]{F(A)}{F(B)}{F(\Comp(\Hilm\oplus B))}
  {\bar{F}(\Hilm)}{F(i_{\Hilm})}{F(f_{\Hilm})}{\bar{F}(u_{\Hilm})}
\end{equation}
in~\(\Cat[D]\).  This gives a recipe to construct \(\bar{F}(\Hilm)\).
If~\(F\) is \(\Cst\)\nb-stable, then \(F(i_{\Hilm})\) is an
equivalence in~\(\Cat[D]\).  We may use a quasi-inverse of it to
complete the pair of maps \(F(f_{\Hilm})\), \(F(i_{\Hilm})\) to a
triangle as in~\eqref{eq:triangle_to_fill}.  Hence we
get~\(\bar{F}(\Hilm)\) for all arrows in~\(\Corr_\prop\), together
with triangles~\(\bar{F}(u_{\Hilm})\).

Next we use the functors \(\Subdiv{n}\to N\Cstcat_+\) to
construct~\(\bar{F}\) on \(n\)\nb-simplices in \(N\Corr_\prop\).  This
is a more complicated version of the triangle filling construction above
with several intermediate steps, which require some bookkeeping to
make sure that each step is possible and that they all taken together
well-define a simplicial map.

Let~\(A_i\) for \(0\le i\le n\), \(\Hilm_{ij}\colon A_i\to A_j\) for
\(0\le i\le j\le n\) and
\(u_{ijk}\colon \Hilm_{ij}\otimes_{A_j} \Hilm_{jk}\to \Hilm_{ik}\) for
\(0\le i\le j\le k\le n\) give an \(n\)\nb-simplex
in~\(N\Corr_\prop\).  Lemma~\ref{lem:fST_functor_Subdiv} shows that
the construction above the lemma yields a simplicial map
\[
  G = G(A_i,\Hilm_{ij},u_{ijk})\colon \Subdiv{n} \to N\Cstcat_+
  \xrightarrow{F} \Cat[D].
\]
In order to define~\(\bar{F}\) on our given \(n\)\nb-simplex, we
extend~\(G\) to a larger simplicial set~\(\CSubdiv{n}\), namely, the
nerve of the partially ordered set that we get from~\(\Subs{n}\) by
adjoining the relations \(\{i\} \le \{j\}\) for all
\(0\le i\le j\le n\).  Thus~\(\CSubdiv{n}\) has \(\ell\)\nb-simplices
\[
  (i_0,i_1,\dotsc ,i_k,S_{k+1},\dotsc ,S_\ell)
\]
for all \(0\le i_0\le i_1\le \cdots\le i_k\le n\) and
\(S_{k+1}\subseteq \cdots\subseteq S_\ell\subseteq [n]\) with
\(i_0,\dotsc ,i_k\in S_{k+1}\) and \(\abs{S_{k+1}}\ge 2\),
\(-1\le k\le \ell\); the cases \(k=-1\) or \(k=\ell\) are allowed.
Those simplices with \(k\le 0\) already belong to~\(\Subdiv{n}\).  As
usual, the \(j\)th face map leaves out the \(j\)th entry in the list,
and the \(j\)th degeneracy map doubles the \(j\)th entry.  Thus a
simplex is nondegenerate if \(0\le i_0<i_1<\cdots<i_k\le n\) and
\[
  \{i_0,\dotsc ,i_k\} \subseteq S_{k+1} \subsetneq S_{k+2} \subsetneq
  \cdots \subsetneq S_\ell \subseteq [n].
\]
We declare \(S_k\defeq \{i_0,\dotsc ,i_k\}\), so that~\(S_\ell\) is
defined also for \(k=\ell\).

We denote the desired extension of~\(G\) by
\[
  \bar{G} = \bar{G}(A_i,\Hilm_{ij},u_{ijk})\colon \CSubdiv{n}\to
\Cat[D].
\]
We will only use the value of~\(\bar{G}\) on the \(n\)\nb-simplex
\((0,\dotsc ,n)\) in \(\CSubdiv{n}\) to define~\(\bar{F}\); the values
on the other simplices of~\(\CSubdiv{n}\) are needed to construct this
value and to relate it to the given functor~\(F\).

We want the maps~\(\bar{G}\) for different simplices
in~\(N\Corr_\prop\) to be compatible with the face and degeneracy maps
of~\(N\Corr_\prop\) in the following way.  Let \(\varphi\colon [m]\to
[n]\) be an order-preserving map.  This maps the \(n\)\nb-simplex
\((A_i,\Hilm_{ij},u_{ijk})\) to an \(m\)\nb-simplex
\[
  (A_{\varphi(i)},\Hilm_{\varphi(i),\varphi(j)},u_{\varphi(i),\varphi(j),\varphi(k)}),
\]
and it induces a simplicial map
\[
  \varphi_*\colon \CSubdiv{m}\to \CSubdiv{n}.
\]
We want
\begin{equation}
  \label{eq:barG_compatible_with_simplicial_maps}
  \bar{G}(A_{\varphi(i)},\Hilm_{\varphi(i)\varphi(j)},u_{\varphi(i)\varphi(j)\varphi(k)})
  = \bar{G}(A_i,\Hilm_{ij},u_{ijk})\circ \varphi_*.
\end{equation}
If the given simplex~\((A_i,\Hilm_{ij},u_{ijk})\) is degenerate,
then~\eqref{eq:barG_compatible_with_simplicial_maps} for surjective
maps~\(\varphi\)
determines~\(\bar{G}\) by its values on simplices in~\(N\Corr_\prop\)
of smaller dimension.  Since we are going to construct~\(\bar{G}\)
inductively over the dimension of the simplices, there is nothing to
do for degenerate simplices and we may assume from now on that
\((A_i,\Hilm_{ij},u_{ijk})\) is nondegenerate.  In addition, it then
suffices to check~\eqref{eq:barG_compatible_with_simplicial_maps} for
injective maps.

The condition~\eqref{eq:barG_compatible_with_simplicial_maps}
determines
\begin{equation}
  \label{eq:barG_i0_ik_Sk_Sl}
  \bar{G}(i_0,\dotsc ,i_k,S_{k+1},\dotsc,S_\ell)
\end{equation}
if \(S_\ell\neq [n]\) because then there is \(\varphi\colon [n-1]\to
[n]\) with \((i_0,\dotsc ,S_\ell)\in \varphi_*(\CSubdiv{n-1})\).
Since~\(\bar{G}\) must extend~\(G\), we already
know~\eqref{eq:barG_i0_ik_Sk_Sl} if \(k\le 0\).
These simplices form a subcomplex of~\(\CSubdiv{n}\).
The prescription for~\(\bar{G}\) on this subcomplex is a simplicial
map.
We must extend it to~\(\CSubdiv{n}\).

It remains to define~\eqref{eq:barG_i0_ik_Sk_Sl} for simplices
in~\(\CSubdiv{n}\) with \(k\ge 1\) and \(S_\ell=[n]\).  We do this by
a recursion through \(k=1,\dotsc ,n\).  When
constructing~\eqref{eq:barG_i0_ik_Sk_Sl} for \(k\ge 1\), we assume
that~\(\bar{G}\) is already given on simplices in~\(\CSubdiv{n}\) with
smaller value of~\(k\) (which is indeed so for \(k=1\)).  For
fixed~\(k\), we run another recursion on~\(\ell\).  In each step, we
fill a horn to construct~\eqref{eq:barG_i0_ik_Sk_Sl} if
\(S_{k+1}=\{i_0,\dotsc ,i_k\}\) and
\(S_{k+1}\subsetneq S_{k+2}\subsetneq \cdots\subsetneq S_\ell=[n]\),
assuming that~\eqref{eq:barG_i0_ik_Sk_Sl} has been constructed for
smaller values of~\(\ell\).

Assume first that \(S_{k+1}=\{i_0,\dotsc ,i_k\}\).  The \(j\)th face
of
\begin{equation}
  \label{eq:i0_ik_Sk_Sl}
  (i_0,\dotsc ,i_k,S_{k+1},\dotsc ,S_\ell)
\end{equation}
for \(0\le j\le k\) has a smaller value of~\(k\), so that~\(\bar{G}\)
on this face has been constructed.  The \(j\)th face
of~\eqref{eq:i0_ik_Sk_Sl} for \(k+2\le j\le \ell\) has a smaller value
of~\(\ell\) and still contains \(S_{k+1}=\{i_0,\dotsc ,i_k\}\), so
that~\(\bar{G}\) on this face has been constructed.  The values
of~\(\bar{G}\) on these boundary faces give an \((\ell,k+1)\)-horn
in~\(\Cat[D]\).  If \(\ell>k+1\), this is an inner horn and so it may
be filled in the quasi-category~\(\Cat[D]\).  Let \(\ell=k+1\).  The
arrow \(i_k\to S_{k+1} = \{i_0,\dotsc ,i_k\}\) induces a corner
embedding, and~\(G\) maps this to an equivalence in~\(\Cat[D]\) by
assumption; hence we have got a special \((\ell,\ell)\)-horn, and
these may be filled in all quasi-categories (see
\cite{Joyal:Quasi-categories}*{Theorem~1.3}).  (By the way, the case
\(\ell=k+1\) only occurs if \(k=n\) because \(S_\ell=[n]\).)

The above recursion for fixed~\(k\) gives~\eqref{eq:barG_i0_ik_Sk_Sl}
for all nondegenerate simplices in~\(\CSubdiv{n}\) with
\(S_{k+1}=\{i_0,\dotsc ,i_k\}\) and \(S_\ell=[n]\).  It also
gives~\(\bar{G}\) on the missing \(k+1\)st face
\[
  (i_0,\dotsc ,i_k,S_{k+2},\dotsc ,S_\ell)
\]
of~\eqref{eq:i0_ik_Sk_Sl}.  Any nondegenerate \(\ell-1\)-simplex
\[
  (i_0,\dotsc ,i_k,T_{k+1},\dotsc ,T_{\ell-1})
\]
in~\(\CSubdiv{n}\) with \(T_{\ell-1}=[n]\) and \(T_{k+1}\neq
\{i_0,\dotsc ,i_k\}\) is of this form for a unique \((i_0,\dotsc
,S_\ell)\), namely,
\[
  (i_0,\dotsc ,i_k,\{i_0,\dotsc ,i_k\},T_{k+1},\dotsc ,T_{\ell-1}).
\]
Hence our recursion defines~\(\bar{G}(i_0,\dotsc ,S_\ell)\) for all
nondegenerate simplices in~\(\CSubdiv{n}\) with \(S_\ell=[n]\).  We
extend the map to degenerate simplices in the unique way dictated by
the degeneracy maps and take this together with the values already
known for \(S_\ell\subsetneq [n]\).  This defines~\(\bar{G}\) on all
simplices of~\(\CSubdiv{n}\).  Since~\(\bar{G}\)
satisfies~\eqref{eq:barG_compatible_with_simplicial_maps},
\[
  \bar{F}(A_i,\Hilm_{ij},u_{ijk}) \defeq
  \bar{G}(A_i,\Hilm_{ij},u_{ijk})(0,1,\dotsc ,n)
\]
defines a simplicial map \(\bar{F}\colon N\Corr_\prop\to \Cat[D]\)
with \(\bar{F}\circ \Gamma=F\).  This finishes the proof of
Proposition~\ref{pro:extend_to_Corrprop}.

\medskip

Let us examine the construction for \(n=1\) and \(n=2\) to understand
better what is happening.

For \(n=1\), there is only one recursion step \(k=1\), \(\ell=2\),
which requires one \((2,2)\)-horn to be filled.  This is exactly the
construction of the triangle in~\eqref{eq:triangle_to_fill}.  The
value~\(\bar{F}(\Hilm)\) on the edge~\(\Hilm\) in~\(N\Corr_\prop\) is
indicated in~\eqref{eq:triangle_to_fill}.

For \(n=2\), Lemma~\ref{lem:fST_functor_Subdiv} defines~\(\bar{G}\) on
the six triangles that make up the barycentric subdivision
\(\Subdiv{n}\) of a triangle, see the top diagram in
Figure~\ref{fig:construction_figures}.
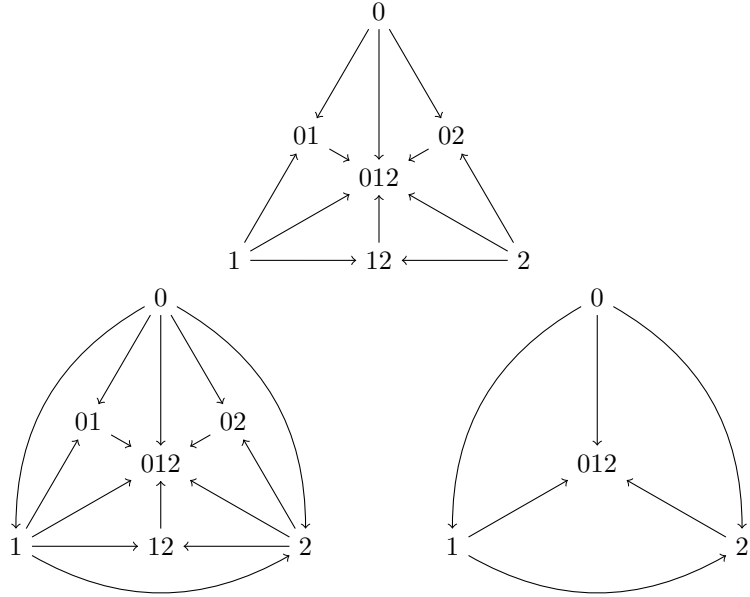
\begin{figure}[htbp]
  \centering
  \begin{tikzpicture}[scale=2.2, baseline=(current bounding box.west)]
    \node (N012) at (0,0) {\(012\)};
    \node (N0) at (90:1) {\(0\)};
    \node (N1) at (210:1) {\(1\)};
    \node (N2) at (330:1) {\(2\)};
    \node (N01) at ($(N0)!0.5!(N1)$) {\(01\)};
    \node (N02) at ($(N0)!0.5!(N2)$) {\(02\)};
    \node (N12) at ($(N2)!0.5!(N1)$) {\(12\)};
    \draw[cdar] (N0)--(N01);
    \draw[cdar] (N0)--(N02);
    \draw[cdar] (N0)--(N012);
    \draw[cdar] (N1)--(N01);
    \draw[cdar] (N1)--(N12);
    \draw[cdar] (N1)--(N012);
    \draw[cdar] (N2)--(N02);
    \draw[cdar] (N2)--(N12);
    \draw[cdar] (N2)--(N012);
    \draw[cdar] (N01)--(N012);
    \draw[cdar] (N02)--(N012);
    \draw[cdar] (N12)--(N012);
  \end{tikzpicture}\\
  \begin{tikzpicture}[scale=2.2, baseline=(current bounding box.west)]
    \node (N012) at (0,0) {\(012\)};
    \node (N0) at (90:1) {\(0\)};
    \node (N1) at (210:1) {\(1\)};
    \node (N2) at (330:1) {\(2\)};
    \node (N01) at ($(N0)!0.5!(N1)$) {\(01\)};
    \node (N02) at ($(N0)!0.5!(N2)$) {\(02\)};
    \node (N12) at ($(N2)!0.5!(N1)$) {\(12\)};
    \draw[cdar] (N0)--(N01);
    \draw[cdar] (N0)--(N02);
    \draw[cdar] (N0)--(N012);
    \draw[cdar] (N1)--(N01);
    \draw[cdar] (N1)--(N12);
    \draw[cdar] (N1)--(N012);
    \draw[cdar] (N2)--(N02);
    \draw[cdar] (N2)--(N12);
    \draw[cdar] (N2)--(N012);
    \draw[cdar] (N01)--(N012);
    \draw[cdar] (N02)--(N012);
    \draw[cdar] (N12)--(N012);
    \draw[cdar] (N0) to [bend right] (N1);
    \draw[cdar] (N0) to [bend left] (N2);
    \draw[cdar] (N1) to [bend right] (N2);
  \end{tikzpicture}
  \qquad\qquad
  \begin{tikzpicture}[scale=2.2, baseline=(current bounding box.west)]
    \node (N012) at (0,0) {\(012\)};
    \node (N0) at (90:1) {\(0\)};
    \node (N1) at (210:1) {\(1\)};
    \node (N2) at (330:1) {\(2\)};
    \draw[cdar] (N0) to [bend right] (N1);
    \draw[cdar] (N0) to [bend left] (N2);
    \draw[cdar] (N1) to [bend right] (N2);
    \draw[cdar] (N0)--(N012);
    \draw[cdar] (N1)--(N012);
    \draw[cdar] (N2)--(N012);
  \end{tikzpicture}
  \caption{The simplicial sets that occur in the construction for $n=2$ }
  \label{fig:construction_figures}
\end{figure}
The construction for \(1\)\nb-simplices (\(n=1\)) has already
defined~\(\bar{G}\) on the three triangles
\(i_0 \to i_1 \to \{i_0,i_1\}\) for
\((i_0,i_1) \in \{(0,1), (0,2), (1,2)\}\) that are added in the bottom
left diagram in Figure~\ref{fig:construction_figures}.  The three
60$^\circ$ corners in this figure give the \((3,2)\)-horns in the
tetrahedra \(i_0 < i_1 < \{i_0,i_2\} < \{0,1,2\}\) for
\((i_0,i_1) \in \{(0,1), (0,2), (1,2)\}\).  Our construction for
\(n=2\), \(k=1\) fills these horns.  This also gives values
for~\(\bar{G}\) on the \(2\)\nb-boundary faces
\(i_0 < i_1 < \{0,1,2\}\) of these tetrahedra.  These triangles are
shown in the bottom right figure of
Figure~\ref{fig:construction_figures}.  They are mapped by~\(\bar{G}\)
to a \((3,3)\)-horn in~\(\Cat[D]\).  This horn is special because the
\Star{}homomorphism \(A_{\{2\}} \to A_{\{0,1,2\}}\), which is
associated to the map \(2\to 012\), is a corner embedding.  Therefore,
this horn in~\(\Cat[D]\) may be filled.  This is done in the recursion
step for \(k=2\), which gives the value of~\(\bar{G}\) on the triangle
\(0\to 1 \to 2\).  This is the value of~\(\bar{F}\) on the given
triangle in~\(N\Corr_\prop\).

\medskip

Let \(m\in\N\).  Since \(\Map(\partial\Simp{m},\Cat[D])\) is a
quasi-category as well and the internal hom is adjoint to the product
of simplicial sets, Proposition~\ref{pro:extend_to_Corrprop} implies
that any simplicial map
\(f\colon N\Cstcat_+\times \partial\Simp{m}\to \Cat[D]\) factors as
\(F\circ (\Gamma\times \id_{\partial\Simp{m}})\) for a simplicial map
\(F\colon N\Corr_\prop\times \partial\Simp{m}\to \Cat[D]\).  That is,
the map~\(\Gamma\) induces surjections on all homotopy groups.  To
prove that~\(\Gamma\) is a homotopy equivalence, it remains to prove
that it also induces injective maps on all homotopy groups.  To prove
this, we show that, given simplicial maps
\[
  h\colon N\Cstcat_+\times \Simp{m}\to \Cat[D],\qquad \partial
  H\colon N\Corr_\prop\times \partial\Simp{m}\to \Cat[D]
\]
with
\(h\circ (\id\times \iota) = \partial H\circ (\Gamma\times \id)\),
where \(\iota\colon \partial \Simp{m}\to \Simp{m}\) is the inclusion
map, there is a simplicial map
\(H\colon N\Corr_\prop\times \Simp{m}\to \Cat[D]\) with
\(H\circ (\id\times \iota) = \partial H\) and
\(H\circ (\Gamma\times \id) = h\).

The construction of~\(H\) is a variant of the one above.  We have to
extend a given map from
\(N\Corr_\prop\times \partial \Simp{m}\cup \Gamma(N\Cstcat_+)\times
\Simp{m}\) to~\(\Cat[D]\) to a map defined on \(N\Corr_\prop\times
\Simp{m}\); we do
this by defining the values on \(n\)\nb-simplices by induction
over~\(n\).  We only construct something new for a nondegenerate
\(n\)\nb-simplex~\(s\) in \(N\Corr_\prop\times \Simp{m}\) that is not
contained in
\(N\Corr_\prop\times \partial \Simp{m}\cup \Gamma(N\Cstcat_+)\times
\Simp{m}\).  Write \(s=(s_1,s_2)\) with \(n\)\nb-simplices \(s_1\)
and~\(s_2\) in~\(N\Corr_\prop\) and \(\Simp{m}\).  The simplex~\(s_1\)
gives a simplicial map \(G'\colon \Subdiv{n}\to N\Cstcat_+\) as above,
and \((G',s_2)\colon \Subdiv{n}\to N\Cstcat_+\times \Simp{m}\) may be
composed with~\(h\) to get a simplicial map
\(G\colon \Subdiv{n}\to \Cat[D]\).  As above, we extend this to a map
\(\bar{G}\colon \CSubdiv{n}\to \Cat[D]\), such that~\(\bar{G}\) as a
function of~\(s\) is compatible with the face maps.  This construction
is exactly the same extension problem as above.  Having
constructed~\(\bar{G}\), we let \(H(s) \defeq \bar{G}(0,\dotsc,n)\).
This defines a simplicial map
\(H\colon N\Corr_\prop\times \Simp{m}\to \Cat[D]\) with the required
properties.

What we have just proved says that~\(\Gamma_*\) induces an isomorphism
on all homotopy groups.  Thus~\(\Gamma_*\) is a
homotopy equivalence by general results in homotopy theory.

% We show how this fact allows us to construct a \emph{simplicial} map
% \[
%   \tilde{B}\colon \Map_\Comp(N\Cstcat_+,\Cat[D]) \to
% \Map(N\Corr_\prop,\Cat[D])
% \]
% with \(\Gamma^* \circ \tilde{B} = \id\).  We define~\(\tilde{B}\) on
% \(m\)\nb-simplices by induction on~\(n\), starting with the map~\(B\)
% for \(m=0\) constructed above.  An \(m\)\nb-simplex
% in~\(\Map_\Comp(N\Cstcat_+,\Cat[D])\) is a simplicial map
% \(h\colon N\Cstcat_+\times \Simp{m}\to \Cat[D]\) that is
% \(\Cst\)\nb-stable on \(N\Cstcat_+\times \Simp{0}\) for each vertex
% in~\(\Simp{m}\).  Since we have constructed~\(\tilde{B}\) on simplices
% of lower dimension, the restriction of~\(h\) to~\(\partial \Simp{m}\)
% is already lifted to a simplicial map
% \(\partial H\colon \Corr\times \partial \Simp{m}\to \Cat[D]\).  From
% the pair of maps \((\partial H,h)\) we construct a simplicial map
% \(H\colon N\Corr_\prop\times \Simp{m}\to \Cat[D]\) and define
% \(\tilde{B}(h) \defeq H\).  This yields the desired simplicial
% map~\(\tilde{B}\).

% The map~\(\tilde{B}\) shows that \(\Map_\Comp(N\Cstcat_+,\Cat[D])\) is
% a retract of \(\Map(N\Corr_\prop,\Cat[D])\).  In fact, it is even a
% \emph{deformation} retract, that is, \(\tilde{B} \circ \Gamma^*\) is
% homotopic to the identity map on \(\Map(N\Corr_\prop,\Cat[D])\).  This
% homotopy is constructed in the same way as~\(\tilde{B}\).

This finishes the proof of \longref{Theorem}{the:Corr_prop_universal}.

\section{Some variants of the main theorem}
\label{sec:variants}

We have interpreted the proper correspondence bicategory as a
Dwyer--Kan localisation of the category of \(\Cst\)\nb-algebras and
\Star{}homomorphisms.  There is a similar interpretation for the
bicategory of all correspondences, proper or not.

An \(A,B\)\nb-correspondence with underlying Hilbert module~\(B\) is
the same as a nondegenerate \Star{}homomorphism from~\(A\) to the
multiplier algebra of~\(B\), briefly called a \emph{morphism}
from~\(A\) to~\(B\).  To formulate \(\Cst\)\nb-stability, we need the
corner embeddings \(i_{\Hilm}\colon B\to \Comp(\Hilm\oplus B)\), which
are degenerate.  Hence \(\Cst\)\nb-stability makes no sense for
functors on the morphism category~\(\Cstcat\): we must allow certain
degenerate morphisms.

A convenient choice is to take the \emph{strictly continuous}
\Star{}homomorphisms \(A\to \Mult(B)\).  These are those
\Star{}homomorphisms \(\varphi\colon A\to \Mult(B)\) that extend to a
strictly continuous \Star{}homomorphism \(\bar{\varphi}\colon
\Mult(A)\to \Mult(B)\); here we use that~\(\Mult(A)\) is the
completion of~\(A\) in the strict topology
(see~\cite{Busby:Double_centralizer}).  Therefore, the strictly
continuous \Star{}homomorphisms \(A\to \Mult(B)\) still form a
category, which we denote by~\(\Cstcat_\strict\).

The extension \(\bar{\varphi}\colon \Mult(A) \to \Mult(B)\) maps
\(1\in \Mult(A)\) to a projection \(p\in \Mult(B)\), and
\(\bar{\varphi}\colon \Mult(A)\to p\Mult(B)p=\Mult(p B p)\) is a
unital, strictly continuous \Star{}homomorphism.  Thus
\(\varphi(A)\subseteq \Mult(pBp)\) and~\(\varphi\) is nondegenerate as
a \Star{}homomorphism to \(\Mult(p B p)\).  A \(\Cst\)\nb-subalgebra
of~\(B\) of the form~\(p B p\) is called a \emph{corner} in~\(B\).
The strictly continuous \Star{}homomorphisms \(A\to \Mult(B)\) are
therefore the same as the morphisms from~\(A\) to corners in~\(B\).

An arrow~\(\varphi\) in~\(\Cstcat_\strict\) gives a correspondence,
namely, \(p B\subseteq B\) with the nondegenerate left \(A\)\nb-action
given by~\(\varphi\).  As for~\(\Cstcat_+\), this identifies the
nerve~\(N\Cstcat_\strict\) with a subcategory in~\(N\Corr\).  The
homomorphisms \(i_{\Hilm}\) and~\(f_{\Hilm}\) for a
correspondence~\(\Hilm\) in
\longref{Lemma}{lem:subdivision_back_to_corr} and the maps
\(f_{S T}\colon A_S \to A_T\) for an \(n\)\nb-simplex in \(N\Corr\)
are strictly continuous.  With this extra fact, the proofs above carry
over and show:

\begin{theorem}
  \label{the:Corr_universal}
  Let~\(\Cat[D]\) be a quasi-category.  The following simplicial map
  is a homotopy equivalence:
  \[
    \Gamma^*\colon \Map(N\Corr, \Cat[D]) \to
    \Map_\Comp(N\Cstcat_\strict,\Cat[D]),\quad
    F\mapsto F\circ \Gamma.
  \]
\end{theorem}

We also get a variant of \longref{Theorem}{the:Corr_prop_universal},
where we replace the category~\(\Cstcat_+\) by its subcategory of all
strictly continuous \Star{}homomorphisms \(A\to B\).

An \(A,B\)-correspondence~\(\Hilm\) is called \emph{full} if the inner
products \(\braket{\xi}{\eta}\) for \(\xi,\eta\in\Hilm\) span~\(B\).
The full correspondences form a subbicategory of the correspondence
bicategory.  We call a strictly continuous homomorphism
\(\varphi\colon A\to\Mult(B)\) \emph{full} if the projection
\(p=\bar\varphi(1)\) is full, that is, the Hilbert module~\(p B\) is
full.  Equivalently, the closed linear span of~\(B p B\) is~\(B\).
This is a subcategory of~\(\Cstcat_\strict\), which we denote
by~\(\Cstcat_\full\).

The homomorphisms \(i_{\Hilm}\) and~\(f_{\Hilm}\) for a
correspondence~\(\Hilm\) in
\longref{Lemma}{lem:subdivision_back_to_corr} are full if~\(\Hilm\) is
full as a right Hilbert module, and the maps
\(f_{S T}\colon A_S \to A_T\) for an \(n\)\nb-simplex in \(N\Corr\)
are full if all~\(\Hilm_{i j}\) are full as right Hilbert modules.
Using this, the same arguments as above show that the simplicial map
\[
  \Gamma^*\colon \Map(N\Corr_\full, \Cat[D]) \to
  \Map_\Comp(N\Cstcat_\full,\Cat[D]),\quad
  F\mapsto F\circ \Gamma,
\]
is a homotopy equivalence.

For the bicategory \(\Corr_\full\cap \Corr_\prop\) of full, proper
correspondences, we get two versions of
\longref{Theorem}{the:Corr_prop_universal}: it is a Dwyer--Kan
localisation of the category of full \Star{}homomorphisms \(A\to B\)
or of the category of full, strictly continuous \Star{}homomorphisms
\(A\to B\).  Here a \Star{}homomorphism \(\varphi\colon A\to B\) is
called \emph{full} if \(B\varphi(A)B\) spans~\(B\); equivalently, the
correspondence~\(\Hilm_\varphi\) is full.

\end{document}